\documentclass[english]{article}
\usepackage[T1]{fontenc}
\usepackage[latin9]{inputenc}
\usepackage{float}
\usepackage{amsmath}
\usepackage{amsthm}
\usepackage{amssymb}
\usepackage{graphicx}

\makeatletter
\theoremstyle{plain}
\newtheorem{thm}{\protect\theoremname}
  \theoremstyle{definition}
  \newtheorem{defn}[thm]{\protect\definitionname}
  \theoremstyle{remark}
  \newtheorem{rem}[thm]{\protect\remarkname}
  \theoremstyle{plain}
  \newtheorem{cor}[thm]{\protect\corollaryname}
  \theoremstyle{plain}
  \newtheorem{lem}[thm]{\protect\lemmaname}
  \theoremstyle{plain}
  \newtheorem{prop}[thm]{\protect\propositionname}
  \theoremstyle{plain}
  \newtheorem{question}[thm]{\protect\questionname}

\usepackage{fullpage}
\usepackage{hyperref}
\hypersetup{
  colorlinks   = true, 
  urlcolor     = red, 
  linkcolor    = blue, 
  citecolor   = blue 
}
\date{}

\makeatother

\usepackage{babel}
  \providecommand{\corollaryname}{Corollary}
  \providecommand{\definitionname}{Definition}
  \providecommand{\lemmaname}{Lemma}
  \providecommand{\propositionname}{Proposition}
  \providecommand{\questionname}{Question}
  \providecommand{\remarkname}{Remark}
\providecommand{\theoremname}{Theorem}

\begin{document}
\global\long\def\goinf{\rightarrow\infty}
\global\long\def\gozero{\rightarrow0}
\global\long\def\bra{\langle}
\global\long\def\ket{\rangle}
\global\long\def\union{\cup}
\global\long\def\intersect{\cap}
\global\long\def\abs#1{\left|#1\right|}
\global\long\def\norm#1{\left\Vert #1\right\Vert }
\global\long\def\floor#1{\left\lfloor #1\right\rfloor }
\global\long\def\ceil#1{\left\lceil #1\right\rceil }
\global\long\def\expect{\mathbb{E}}
\global\long\def\e{\mathbb{E}}
\global\long\def\r{\mathbb{R}}
\global\long\def\n{\mathbb{N}}
\global\long\def\q{\mathbb{Q}}
\global\long\def\c{\mathbb{C}}
\global\long\def\z{\mathbb{Z}}
\global\long\def\grad{\nabla}
\global\long\def\t{\mathbb{T}}
\global\long\def\all{\forall}
\global\long\def\eps{\varepsilon}
\global\long\def\quadvar#1{V_{2}^{\pi}\left(#1\right)}
\global\long\def\cal#1{\mathcal{#1}}
\global\long\def\cross{\times}
\global\long\def\del{\nabla}
\global\long\def\parx#1{\frac{\partial#1}{\partial x}}
\global\long\def\pary#1{\frac{\partial#1}{\partial y}}
\global\long\def\parz#1{\frac{\partial#1}{\partial z}}
\global\long\def\part#1{\frac{\partial#1}{\partial t}}
\global\long\def\partheta#1{\frac{\partial#1}{\partial\theta}}
\global\long\def\parr#1{\frac{\partial#1}{\partial r}}
\global\long\def\curl{\nabla\times}
\global\long\def\rotor{\nabla\times}
\global\long\def\one{\mathbf{1}}
\global\long\def\Hom{\text{Hom}}
\global\long\def\pr#1{\text{Pr}\left[#1\right]}
\global\long\def\almost{\mathbf{\approx}}
\global\long\def\tr{\text{Tr}}
\global\long\def\var{\text{Var}}
\global\long\def\onenorm#1{\left\Vert #1\right\Vert _{1}}
\global\long\def\twonorm#1{\left\Vert #1\right\Vert _{2}}
\global\long\def\Inj{\mathfrak{Inj}}
\global\long\def\inj{\mathsf{inj}}

\global\long\def\g{\mathfrak{\cal G}}
\global\long\def\f{\mathfrak{\cal F}}

\title{Brownian motion can feel the shape of a drum}

\author{Renan Gross\thanks{Weizmann Institute of Science. Email: renan.gross@weizmann.ac.il.
Supported by the Adams Fellowship Program of the Israel Academy of
Sciences and Humanities.}}
\maketitle
\begin{abstract}
We study the scenery reconstruction problem on the $d$-dimensional
torus, proving that a criterion on Fourier coefficients obtained by
Matzinger and Lember (2006) for discrete cycles applies also in continuous
spaces. In particular, with the right drift, Brownian motion can be
used to reconstruct any scenery. To this end, we prove an injectivity
property of an infinite Vandermonde matrix. \\
\textbf{}\\
\textbf{Keywords}: Scenery reconstruction problem; infinite Vandermonde
matrix; Brownian motion.
\end{abstract}

\section{Introduction}

\subsection{Background}

In its most general formulation, the scenery reconstruction problem
asks the following: Let $C$ be a set, let be $f$ a function on $C$,
and $\left(X_{t}\right)_{t\geq0}$ a stochastic process taking values
in $C$. What information can we learn about $f$ from the (infinite)
trace $f\left(X_{t}\right)_{t\geq0}$? Can $f$ be completely reconstructed
from this trace?

In one of the most common settings, $C$ is taken to be the discrete
integer graph $\z$, the function $f$ maps $C$ to $\left\{ 0,1\right\} $,
and $X_{t}$ is a discrete-time random walk. For this model, numerous
results exist in the literature for a variety of cases, e.g reconstruction
when $f$ is random \cite{benjamini_kesten_distinguishing_sceneries}
and when $f$ is periodic. In the latter case, $f$ is essentially
defined on a cycle of length $\ell$. Matzinger and Lember showed
the following:
\begin{thm}[{\cite[Theorem 3.2]{matzinger_lember_reconstruction_of_periodic_sceneries_seen_along_a_random_walk}}]
\label{thm:matzinger_lember}Let $f$ be a $2$-coloring of the cycle
of length $\ell$, and let $X_{t}$ be a random walk with step distribution
$\gamma\left(x\right)$. If the Fourier coefficients $\left\{ \hat{\gamma}\left(k\right)\right\} _{k=0}^{\ell-1}$
are all distinct, then $f$ can be reconstructed from the trace $f\left(X_{t}\right)$.
\end{thm}

Finucane, Tamuz and Yaari \cite{finucane_tamuz_yaari_scenery_reconstruction_on_finite_abelian_groups}
considered the problem for finite Abelian groups, and showed that
in many such cases, the above condition on the Fourier coefficients
is necessary. 

The problem can also be posed for a continuous space $C$, such as
$\r^{d}$ or the torus $\mathbb{T}^{d}$, with $X_{t}$ a continuous-time
stochastic process. Here there has been considerably less work; to
the best of our knowledge, at the time of writing this paper there
are only two published results: Detecting ``bells'' \cite{matzinger_popov_detecting_a_local_perturbation_in_a_continuous_scenery}
and reconstructing iterated Brownian motion \cite{kryzsztof_some_path_properties_of_iterated_brownian_motion}.
See \cite{kesten_distinguishing_and_reconstructing_sceneries,matzinger_lember_scenery_reconstruction_an_overview}
and references therein for an overview of the reconstruction problem,
with a focus on $\z$ and $\z^{d}$.

\subsection{Results}

In this paper, we extend Theorem \ref{thm:matzinger_lember} from
the discrete cycle to the continuous $d$-dimensional torus $\t^{d}=\left(\r/2\pi\z\right)^{d}$.
The discrete-time random walks are replaced by continuous-time processes
$X_{t}$ (such as Brownian motion), and the $2$-colorings are replaced
by the indicators $f$ of open sets. For an example of how the sample
paths $f\left(X_{t}\right)$ might look like, see Figure \ref{fig:brownian_example},
where $f$ is the indicator of a union of three intervals on the circle,
and $X_{t}$ is Brownian motion. The goal is to reconstruct the size
and position of the intervals, up to rotations, from the trace $f\left(X_{t}\right)$. 

\begin{figure}[H]
\begin{centering}
\includegraphics[width=0.75\paperwidth]{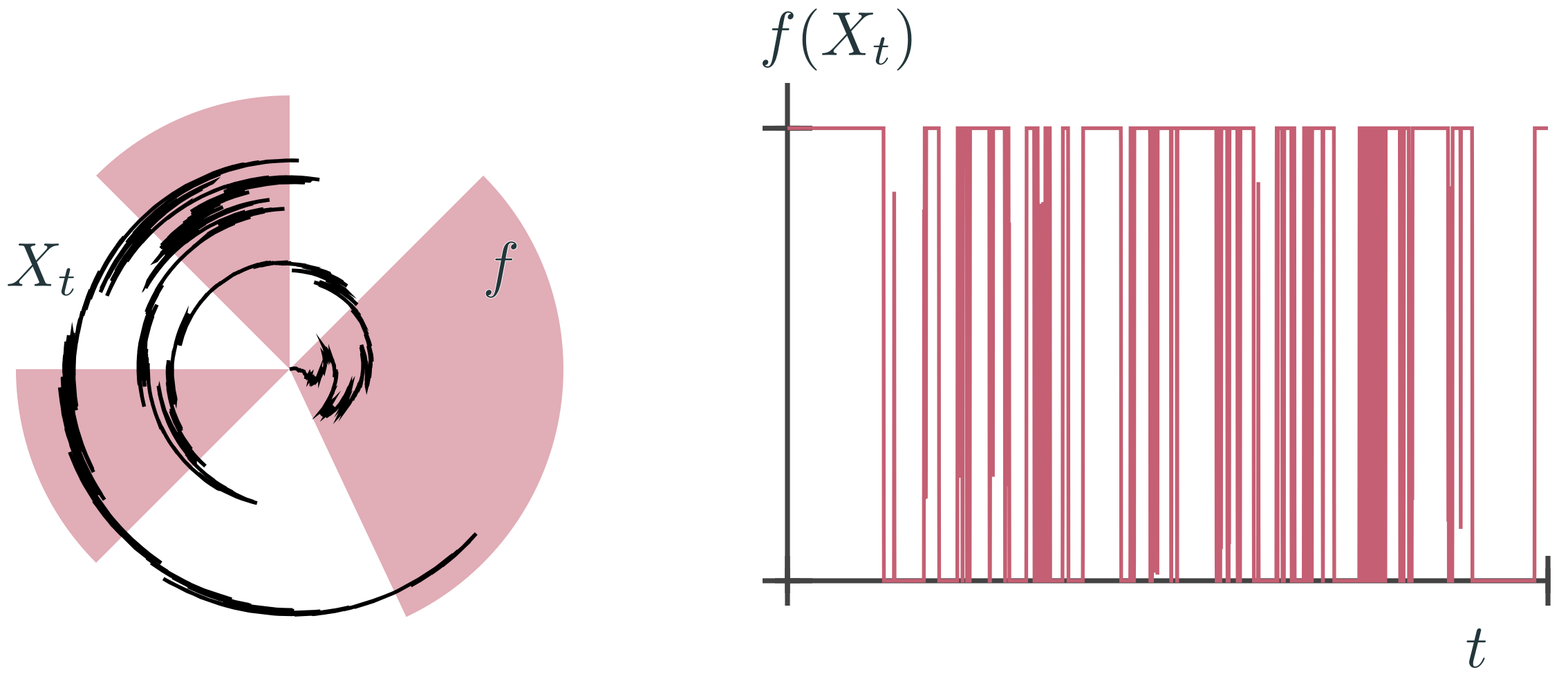}
\par\end{centering}
\caption{\textbf{Left}: Scenery reconstruction in one dimension. The black
curve is a polar depiction of one-dimensional Brownian motion, with
$\theta=X_{t}$ and $r=\sqrt{t}$ (so that points near the center
represent times close $0$, and points near the edge represent larger
times). The function $f$ is represented by the three shaded sectors.
\textbf{Right}: The trace $f\left(X_{t}\right)$. It is equal to $1$
precisely when the curve on the left is inside one of the shaded sectors.\label{fig:brownian_example}}
\end{figure}

\begin{defn}
Let $\mathcal{F}$ be a family of functions on the torus. The family
$\mathcal{F}$ is said to be \emph{reconstructible }by $X_{t}$ if
there is a function $A:\r^{\r^{+}}\to\r^{\mathbb{T}^{d}}$ such that
for every $f\in\mathcal{F}$, with probability $1$ there exists a
(random) shift $\theta\in\t^{d}$ such that $A\left(f\left(X_{t}\right)\right)\left(x\right)=f\left(x+\theta\right)$
for almost all $x$. 
\end{defn}

For the particular case of $d=1$, we also deal with reconstruction
up to reflections:
\begin{defn}
The family $\mathcal{F}$ is said to be \emph{reconstructible up to
reflections} by $X_{t}$, if for all $f\in\mathcal{F}$, with probability
$1$ either $A\left(f\left(X_{t}\right)\right)\left(x\right)=f\left(x+\theta\right)$
for almost all $x$, or $A\left(f\left(X_{t}\right)\right)\left(x\right)=f\left(-x+\theta\right)$
for almost all $x$.
\end{defn}

In order to analyze the trace $f\left(X_{t}\right)$, we must of course
have some control over the behavior of $X_{t}$. In this paper, we
assume that $X_{t}$ is an infinitely-divisible process with independent
increments (this is the natural analog of a discrete-time random walk
with independent steps). That is, there is a time-dependent distribution
$D_{t}$ on $\t^{d}$ such that 
\begin{enumerate}
\item $X_{t_{2}}-X_{t_{1}}\sim D_{t_{2}-t_{1}}$ for every $t_{2}\geq t_{1}$;
\item For all $0\leq t_{1}\leq\ldots\leq t_{n}$, the increments $X_{t_{2}}-X_{t_{1}},\ldots,X_{t_{n}}-X_{t_{n-1}}$
are independent;
\item $D_{t+s}=D_{t}\star D_{s}$ for every $s,t\geq0$, where $\star$
is the convolution operator.
\end{enumerate}
We will also assume that $D_{t}$ is either continuous, or that it
is a mixture of an atom at $0$ and a continuous distribution. In
other words, writing $D_{t}$ as a function of $x$ for simplicity,
we have 
\[
D_{t}\left(x\right)=\beta_{t}\delta\left(x\right)+\left(1-\beta_{t}\right)\gamma_{t}\left(x\right),
\]
where $\delta\left(x\right)$ is the Dirac $\delta$-distribution,
$\gamma_{t}$ is a probability density function, and $\beta_{t}\in\left[0,1\right]$
is a time dependent factor. We will also assume that $\gamma_{t}$
is not too wild: $\gamma_{t}\in\mathrm{L}^{2}\left(\t^{d}\right)$
for all $t>0$. 
\begin{rem}
This class of distributions includes Brownian motion, and any Poisson
process whose steps have an $\mathrm{L}^{2}$ probability density
function. It also contains the sum of Brownian motion and any arbitrary
independent Poisson process, since the diffusion smooths out any irregularities
in the jumps. It does not, however, contain general jump processes
with atoms, even if the atoms are dense in $\t^{d}$ (e.g a Poisson
process on $\t$ which jumps by a step size $\alpha$ rationally independent
from $\pi$).
\end{rem}

The functions we reconstruct will be the indicators of open sets,
whose boundary has $0$ measure in $\r^{d}$. Let
\[
\mathcal{F}_{d}=\left\{ \one_{x\in\Omega}\left(x\right)\mid\Omega\subseteq\t^{d}\text{ is open, }\mathrm{Lebesgue}_{d}\left(\partial\Omega\right)=0\right\} .
\]
Our main result is as follows:
\begin{thm}[General reconstruction]
\label{thm:general_reconstruction}Let $X_{t}$ be a stochastic process
on $\t^{d}$ as above. If there exists a time $t_{0}$ such that the
Fourier coefficients $\left\{ \hat{\gamma}_{t_{0}}\left(k\right)\right\} _{k\in\z^{d}}$
are all distinct and nonzero, then $\mathcal{F}_{d}$ is reconstructible
by $X_{t}$. 
\end{thm}

In one dimension, we show that symmetric distributions can reconstruct
$\mathcal{F}_{1}$ up to reflections:
\begin{thm}[Symmetric reconstruction]
\label{thm:symmetric_reconstruction}Let $X_{t}$ be a stochastic
process on $\t$ as above, and suppose that $\gamma_{t}$ is symmetric,
i.e $\gamma_{t}\left(y\right)=\gamma_{t}\left(-y\right)$ for all
$y$. If there exists a time $t_{0}$ such that the positive-indexed
Fourier coefficients $\left\{ \hat{\gamma}_{t_{0}}\left(k\right)\right\} _{k\geq0}$
are all distinct and nonzero, then $\mathcal{F}_{1}$ is reconstructible
up to reflections by $X_{t}$.
\end{thm}

One corollary of Theorem \ref{thm:general_reconstruction}, is that
with the right drift, Brownian motion can be used to reconstruct $\mathcal{F}_{d}$. 
\begin{cor}[Brownian motion can feel the shape of a drum]
\label{cor:brownian_motion_can_feel}Let $X_{t}$ be Brownian motion
with drift $v\in\r^{d}$, such that $\left\{ v_{1},\ldots,v_{d}\right\} $
are rationally independent. Then $\mathcal{F}_{d}$ is reconstructible
by $X_{t}$. 
\end{cor}

\begin{rem}
The condition on the drift $v$ is natural: If the components of $v$
are rationally independent, then the geodesic flow defined by $v$
is dense in $\t^{d}$. Reconstructing a set $\Omega$ from this geodesic
is immediate. In this sense, Corollary \ref{cor:brownian_motion_can_feel}
states that reconstruction is possible also in the presence of noise
which pushes us out of the trajectory. See Section \ref{sec:open_questions}
for a question on a related model.
\end{rem}

The starting point for our results is a relation, introduced in \cite{matzinger_lember_reconstruction_of_periodic_sceneries_seen_along_a_random_walk},
between two types of $n$-point correlations related to $f$ - one
known, and one unknown. After inverting the relation, the latter correlation
can be used to reconstruct the function $f$. In the discrete case,
the relation is readily inverted using a finite Vandermonde matrix.
In the continuous setting, additional difficulties arise due to both
the more complicated nature of the distribution $D_{t}$, which mixes
together different correlations, and the fact that the function space
on $\t^{d}$ is infinite-dimensional. To address the latter issue,
we prove an injectivity result for infinite Vandermonde matrices,
which may be of independent interest.
\begin{lem}[Infinite Vandermonde]
\label{lem:infinite_vandemonde}Let $p,q\in\left[1,\infty\right]$
be such that $\frac{1}{p}+\frac{1}{q}=1$. Let $\left(z_{n}\right)_{n}\in\ell^{p}\left(\c\right)$
be a sequence of distinct complex numbers such that $z_{n}\to0$,
and $z_{n}\neq0$ for all $n$. Let $V$ be the infinite Vandermonde
matrix with $z_{n}$ as generators, i.e 
\[
V_{ij}=z_{j}^{i}.
\]
If $\boldsymbol{x}\in\ell^{q}\left(\c\right)$ is a zero of the infinite
system of equations
\begin{equation}
V\boldsymbol{x}=0,\label{eq:infinite_matrix_equation}
\end{equation}
then $\boldsymbol{x}=0$.
\end{lem}

\begin{rem}
\label{rem:vandermonde_multiplication_is_ok}The matrix equation $V\boldsymbol{x}=0$
means that for every index $i\in\n$ we have 
\[
0=\sum_{j=1}^{\infty}V_{ij}x_{j}=\sum_{j=1}^{\infty}z_{j}^{i}x_{j}.
\]
By Hölder's inequality, since $z\in\ell^{p}\left(\c\right)$ and $\boldsymbol{x}\in\ell^{q}\left(\c\right)$,
the series $\left(z_{j}^{i}x_{j}\right)_{j}$ is absolutely convergent,
and the left hand side of (\ref{eq:infinite_matrix_equation}) is
well defined. 
\end{rem}

Basic preliminaries and the proof of Theorem \ref{thm:general_reconstruction}
are given in the next section. Section \ref{sec:symmetric_reconstruction}
gives the outline of Theorem \ref{sec:general_reconstruction}, relying
on the same techniques described in Section \ref{sec:general_reconstruction}.
Brownian motion is discussed in Section \ref{sec:example_brownian_motion},
and Lemma \ref{lem:infinite_vandemonde} is proved in Section \ref{sec:vandermonde_lemma}.
We conclude the paper with some open questions.

\subsection{Acknowledgments}

The author thanks Itai Benjamini, Ronen Eldan, Shay Sadovsky, Ori
Sberlo and Ofer Zeitouni for their stimulating discussions and suggestions.

\section{General reconstruction\label{sec:general_reconstruction}}

\subsection{Notation and simple properties of $D_{t}$}

We write $d$-dimensional vectors in standard italics, e.g $k\in\z^{d}$.
Tuples of $n$ vectors are written in boldface, e.g $\boldsymbol{k}\in\z^{nd}$,
with $\boldsymbol{k}=\left(k_{1},\ldots,k_{n}\right)$ and $k_{i}\in\z^{d}$. 

The Fourier series of a function $g:\t^{d}\to\c$ is a function $\hat{g}:\z^{d}\to\c$,
given by 
\[
\hat{g}\left(k\right)=\int_{\t^{d}}g\left(x\right)e^{-ik\cdot x}dx.
\]
Note that we do not divide by the customary $1/\left(2\pi\right)^{d}$.
This simplifies the statement of the convolution theorem: In this
setting, we have 
\[
\widehat{f\star g}=\hat{f}\cdot\hat{g},
\]
without any leading factor in the right hand side. 

This definition also extends to the Dirac $\delta$-distribution,
even though it is not a function, so that for all $k\in\z^{d}$,
\[
\hat{\delta}\left(k\right)=\int_{-\pi}^{\pi}\delta\left(x\right)e^{-ik\cdot x}dx=1.
\]
Recall that $D_{t}=\beta_{t}\delta+\left(1-\beta_{t}\right)\gamma_{t}$.
Using the fact that $D_{s+t}=D_{s}\star D_{t}$, a short calculation
shows that if the parameter $\beta_{t}$ is not identically $0$,
it must decay exponentially: $\beta_{t}=e^{-ct}$ for some constant
$c$. We implicitly assume that $\beta_{t}$ is not identically $1$,
as in this case $X_{t}$ does not move and is uninteresting. We thus
have that 
\begin{equation}
\beta_{\alpha t}=\beta_{t}^{\alpha}\,\,\,\,\,\,\,\all\alpha,t>0.\label{eq:beta_is_multiplicative}
\end{equation}
Since $D_{t}$ has a probability density function $\gamma_{t}$, we
have that $D_{t}\to U\left(\t^{d}\right)$ in distribution as $t\to\infty$,
i.e $X_{t}$ converges to the uniform distribution no matter its starting
point. 

The distribution $D_{t}$ has a Fourier representation $\hat{D}_{t}$,
given by 
\begin{equation}
\hat{D}_{t}\left(k\right)=\beta_{t}\hat{\delta}\left(k\right)+\left(1-\beta_{t}\right)\hat{\gamma}_{t}\left(k\right)=\beta_{t}+\left(1-\beta_{t}\right)\hat{\gamma}_{t}\left(k\right)\,\,\,\,\,k\in\z^{d}.\label{eq:fourier_coeff_of_d_t}
\end{equation}
By the convolution theorem, $\hat{D}_{t+s}=\widehat{D_{t}\star D_{s}}=\hat{D}_{t}\hat{D}_{s}$.
From this it follows that for any $\alpha,t>0$, we have 

\[
\hat{D}_{\alpha t}=\left(\hat{D}_{t}\right)^{\alpha}.
\]
Plugging this into (\ref{eq:fourier_coeff_of_d_t}) gives 
\[
\beta_{\alpha t}+\left(1-\beta_{\alpha t}\right)\hat{\gamma}_{\alpha t}\left(k\right)=\left(\beta_{t}+\left(1-\beta_{t}\right)\hat{\gamma}_{t}\left(k\right)\right)^{\alpha},
\]
and so by (\ref{eq:beta_is_multiplicative}),
\begin{equation}
\hat{\gamma}_{\alpha t}=\frac{\left(\beta_{t}+\left(1-\beta_{t}\right)\hat{\gamma}_{t}\right)^{\alpha}-\beta_{t}^{\alpha}}{1-\beta_{t}^{\alpha}}.\label{eq:fourier_coeff_of_gamma_alpha_t}
\end{equation}

\subsection{Proof of Theorem \ref{thm:general_reconstruction}}

The proofs of Theorems \ref{thm:general_reconstruction} and \ref{thm:symmetric_reconstruction}
use the relation between the spatial correlation and temporal correlation
introduced in \cite{matzinger_lember_reconstruction_of_periodic_sceneries_seen_along_a_random_walk}.
Let $f$ be the indicator of an open set $\Omega$. For every integer
$n\geq0$, define the $n$-th spatial correlation of $f$, denoted
$S_{n}:\t^{nd}\to\r$, as 
\begin{align*}
S_{n}\left(\boldsymbol{y}\right) & =\frac{1}{\left(2\pi\right)^{d}}\int_{\t^{d}}f\left(x\right)f\left(x+y_{1}\right)\cdots f\left(x+\sum_{i=1}^{n}y_{i}\right)dx\\
 & =\frac{1}{\left(2\pi\right)^{d}}\int_{\mathbb{T}^{d}}f\left(x\right)\prod_{k=1}^{n}\left[f\left(x+\sum_{i=1}^{k}y_{i}\right)\right]dx,
\end{align*}
 and the $n$-th temporal correlation of $f\left(X_{t}\right)$, denoted
$T_{n}:\r_{+}^{n}\to\r$, as 
\[
T_{n}\left(\boldsymbol{t}\right)=\e_{X_{0}\sim U\left(\mathbb{T}^{d}\right)}\left[f\left(X_{0}\right)f\left(X_{t_{1}}\right)\cdots f\left(X_{\sum_{i=1}^{n}t_{i}}\right)\right]
\]
(note that $S_{0}$ and $T_{0}$ are just constants, equal to the
measure of $\Omega$ relative to $\t^{d}$). The proof involves two
parts: The first shows that $S_{n}\left(\boldsymbol{y}\right)$ can
be calculated from our knowledge of $T_{n}\left(\boldsymbol{t}\right)$.
The second uses $S_{n}\left(\boldsymbol{y}\right)$ to reconstruct
$f$ with better and better precision as $n\to\infty$.
\begin{prop}
\label{prop:reconstructing_s_n}Let $n\in\n$. Under the conditions
of Theorem \ref{thm:general_reconstruction}, the function $S_{n}$
can be calculated from $f\left(X_{t}\right)$ with probability $1$.
\end{prop}

We intend to prove Proposition \ref{prop:reconstructing_s_n} using
the Vandermonde lemma (Lemma \ref{lem:infinite_vandemonde}). To this
end, we first show the following:
\begin{prop}
\label{prop:reconstructing_monomials}For every positive integers
$n,m\in\n$, and every set of times $t_{1},\ldots,t_{n}>0$, the value
of the sum 
\[
\sum_{\boldsymbol{k}\in\z^{nd}}\left(\prod_{i=1}^{n}\hat{\gamma}_{t_{i}}\left(k_{i}\right)\right)^{m}\hat{S}\left(\boldsymbol{k}\right)
\]
 can be calculated from $f\left(X_{t}\right)$ with probability $1$.
\end{prop}

\begin{proof}
By (\ref{eq:fourier_coeff_of_gamma_alpha_t}), for every $m\in\n$
and $t>0$ we have 
\[
\hat{\gamma}_{mt}=\frac{\left(\beta_{t}+\left(1-\beta_{t}\right)\hat{\gamma}_{t}\right)^{m}-\beta_{t}^{m}}{1-\beta_{t}^{m}}.
\]
Rearranging, we can write $\hat{\gamma}_{t}^{m}$ as a sum of smaller
powers of $\hat{\gamma}_{t}$:
\[
\hat{\gamma}_{t}^{m}=\sum_{j=1}^{m-1}c_{j}\hat{\gamma}_{t}^{j}+\hat{\gamma}_{mt},
\]
where $c_{j}$ are some coefficients (in particular, when $D_{t}$
has no atom, i.e when $\beta_{t}=0$ for all $t>0$, this sum is rather
simple: $\hat{\gamma}_{t}^{m}=\hat{\gamma}_{mt}$). Reiterating this
process, we find that the $m$-th power of $\hat{\gamma}_{t}$ can
be written as some linear combination 
\[
\hat{\gamma}_{t}^{m}=\sum_{j=1}^{m-1}c_{j}'\hat{\gamma}_{jt}.
\]
Thus, the product $\left(\prod_{i=1}^{n}\hat{\gamma}_{t_{i}}\left(k_{i}\right)\right)^{m}$
is itself a sum of multilinear monomials in $\left\{ \hat{\gamma}_{jt_{i}}\right\} _{j=1}^{m-1}$,
and so to prove the proposition it suffices to prove it for $m=1$.

The temporal correlation $T_{n}\left(\boldsymbol{t}\right)$ can be
computed by the process $f\left(X_{t}\right)$: Since $X_{t}$ approaches
the uniform distribution on $\t^{d}$ as $t\to0$, for any fixed times
$t_{1},\ldots,t_{n}$, we can choose sampling times $\tau_{1},\tau_{2},\ldots$
so that $\left\{ f\left(X_{\tau_{j}}\right)f\left(X_{\tau_{j}+t_{1}}\right)\ldots f\left(X_{\tau_{j}+\sum_{i=1}^{n}t_{i}}\right)\right\} _{j=1}^{\infty}$
have pairwise correlations that are arbitrarily small, and are arbitrarily
close in distribution to $f\left(X_{0}\right)f\left(X_{t_{1}}\right)\cdots f\left(X_{\sum_{i=1}^{n}t_{i}}\right)$
with $X_{0}\sim U\left(\t^{d}\right)$. The temporal correlation $T_{n}\left(\boldsymbol{t}\right)$
is then given, with probability $1$, by the sample average at times
$\tau_{j}$.

A relation between the spatial and temporal correlation can be obtained
as follows. First, since $X_{0}$ is uniform on $\mathbb{T}^{d}$
in the definition of $T_{n}\left(\boldsymbol{t}\right)$, 
\[
T_{n}\left(\boldsymbol{t}\right)=\frac{1}{\left(2\pi\right)^{d}}\int_{\mathbb{T}^{d}}f\left(x\right)\e\left[\prod_{k=1}^{n}f\left(X_{\sum_{i=1}^{k}t_{i}}\right)\big|X_{0}=x\right]dx.
\]
By conditioning on the event that between times $t_{i-1}$ and $t_{i}$
the process took a step of size $y_{i}$, this is equal to 
\begin{align}
T_{n}\left(\boldsymbol{t}\right) & =\int_{\mathbb{T}^{d}}\int_{\mathbb{T}^{nd}}\left(\prod_{i=1}^{n}D_{t_{i}}\left(y_{i}\right)\right)\left(f\left(x\right)\prod_{k=1}^{n}\left[f\left(x+\sum_{i=1}^{k}y_{i}\right)\right]\right)\boldsymbol{d}\boldsymbol{y}dx\nonumber \\
 & =\int_{\t^{nd}}\left(\prod_{i=1}^{n}D_{t_{i}}\left(y_{i}\right)\right)S_{n}\left(\boldsymbol{y}\right)\boldsymbol{dy}.\label{eq:temporal_spatial_relation}
\end{align}
Since $D_{t}\left(y\right)=\beta_{t}\delta\left(y\right)+\left(1-\beta_{t}\right)\gamma_{t}\left(y\right)$,
the product inside the integral breaks into a sum, where, for each
time step $t_{i}$, we have to choose whether the process stayed in
place (corresponding the $\delta\left(y_{i}\right)$), or moved according
to the density $\gamma_{t_{i}}$. Whenever we choose to stay in place,
we shrink the number of spatial variables in our correlation, since
$S_{n}\left(y_{1},\ldots,y_{k},0,y_{k+2},\ldots,y_{n}\right)=S_{n-1}\left(y_{1},\ldots,y_{k},y_{k+2},\ldots,y_{n}\right)$.
We can thus go over all choices $A\subseteq\left[n\right]$ of indices
of times when the process moved according to $\gamma_{t}$, giving
\[
T_{n}\left(\boldsymbol{t}\right)=\sum_{A\subseteq\left[n\right]}\prod_{i\notin A}\beta_{t_{i}}\prod_{i\in A}\left(1-\beta_{t_{i}}\right)\int_{\t^{\abs Ad}}\left(\prod_{i\in A}\gamma_{t_{i}}\left(y_{i}\right)\right)S_{\abs A}\left(\boldsymbol{y}\text{ restricted to \ensuremath{A}}\right)\left(\prod_{i\in A}dy_{i}\right).
\]
The integral in this expression can be seen as an inner product between
$\prod_{i\in A}\gamma_{t_{i}}\left(y_{i}\right)$ and $\overline{S_{\abs A}\left(\boldsymbol{y}\right)}$
over the torus $\t^{\abs Ad}$. Since both $\gamma_{t}$ and $S_{\abs A}$
are in $\mathrm{L}^{2}$$\left(\t^{d}\right)$, by Parseval's theorem,
we can therefore replace it by a sum over all Fourier coefficients
$\boldsymbol{k}\in\z^{\abs Ad}$:
\begin{equation}
T_{n}\left(\boldsymbol{t}\right)=\frac{1}{\left(2\pi\right)^{d}}\sum_{A\subseteq\left[n\right]}\prod_{i\notin A}\beta_{t_{i}}\prod_{i\in A}\left(1-\beta_{t_{i}}\right)\sum_{\boldsymbol{k}\in\z^{\abs Ad}}\left(\prod_{i\in A}\hat{\gamma}_{t_{i}}\left(k_{i}\right)\right)\hat{S}_{\abs A}\left(\boldsymbol{k}\text{ restricted to \ensuremath{A}}\right).\label{eq:t_n_as_large_sum}
\end{equation}
With the above display, Proposition \ref{prop:reconstructing_monomials}
follows quickly by induction on $n$. For the case $n=1$, we have
\[
\left(2\pi\right)^{d}T_{1}\left(t\right)=\beta_{t}S_{0}+\sum_{k\in\z^{d}}\hat{\gamma}_{t}\left(k\right)\hat{S}_{1}\left(k\right).
\]
Since $S_{0}=T_{0}$, the value of $\sum_{k\in\z^{d}}\hat{\gamma}_{t}\left(k\right)\hat{S}_{1}\left(k\right)$
is known. The induction step for general $n$ is now immediate, since
by (\ref{eq:t_n_as_large_sum}) it is evident that $T_{n}\left(\boldsymbol{t}\right)$
is a multilinear polynomial in $\hat{\gamma}_{t_{i}}$, and all terms
with degree strictly smaller than $n$ are known by the induction
hypothesis.
\end{proof}
\begin{proof}[Proof of Proposition \ref{prop:reconstructing_s_n}]
We will now show that for every $n\in\n$, the values 
\[
\left\{ \sum_{\boldsymbol{k}\in\z^{nd}}\left(\prod_{i=1}^{n}\hat{\gamma}_{t_{i}}\left(k_{i}\right)\right)^{m}\hat{S}_{n}\left(\boldsymbol{k}\right)\mid m\in\n,t_{1},\ldots,t_{n}>0\right\} 
\]
uniquely determine $\hat{S}_{n}\left(\boldsymbol{k}\right)$. Proposition
\ref{prop:reconstructing_s_n} then follows, since (as a quick, omitted,
calculation shows) $S_{n}\left(\boldsymbol{y}\right)$ is continuous
on $\t^{d}$, and is thus completely determined by its Fourier coefficients. 

Suppose that there exists a bounded function $Q_{n}\left(\boldsymbol{y}\right)$
on $\t^{nd}$ such that for all $m\in\n$ and all times $t_{1},\ldots,t_{n}>0$,
\[
\sum_{\boldsymbol{k}\in\z^{nd}}\left(\prod_{i=1}^{n}\hat{\gamma}_{t_{i}}\left(k_{i}\right)\right)^{m}\hat{S}_{n}\left(\boldsymbol{k}\right)=\sum_{\boldsymbol{k}\in\z^{nd}}\left(\prod_{i=1}^{n}\hat{\gamma}_{t_{i}}\left(k_{i}\right)\right)^{m}\hat{Q}_{n}\left(\boldsymbol{k}\right).
\]
Then, denoting $x_{\boldsymbol{k}}=\hat{S}_{n}\left(\boldsymbol{k}\right)-\hat{Q}_{n}\left(\boldsymbol{k}\right)$,
we have that 
\[
\sum_{\boldsymbol{k}\in\z^{nd}}\left(\prod_{i=1}^{n}\hat{\gamma}_{t_{i}}\left(k_{i}\right)\right)^{m}x_{\boldsymbol{k}}=0.
\]
Note that $x_{\boldsymbol{k}}\in\ell^{2}\left(\c\right)$, since both
$S_{n}$ and $Q_{n}$ are in $\mathrm{L}^{2}\left(\t^{d}\right)$.
We wish to use Lemma \ref{lem:infinite_vandemonde} to show that necessarily
$x_{\boldsymbol{k}}=0$.

To do this, we must choose the times $t_{i}$ so that the products
$\prod_{i=1}^{n}\hat{\gamma}_{t_{i}}\left(k_{i}\right)$ are all non-zero
and distinct. Then, setting $z_{\boldsymbol{k}}=\prod_{i=1}^{n}\hat{\gamma}_{t_{i}}\left(k_{i}\right)$,
we will have that $z_{\boldsymbol{k}}\in\ell^{2}\left(\c\right)$
(since $\hat{\gamma}_{t_{i}}$ are the Fourier coefficients of a square
integrable function $\gamma_{t_{i}}$), $z_{\boldsymbol{k}}$ are
all distinct, and $z_{\boldsymbol{k}}\neq0$ for all $k\in\z^{nd}$,
exactly meeting the requirements of the lemma with $p=q=1/2$. 

To choose the times $t_{i}$, recall that by assumption, there exists
a time $t_{0}$ such that $\left\{ \hat{\gamma}_{t_{0}}\left(k\right)\right\} _{k\in\z^{d}}$
are all distinct and nonzero. We will show that there exist numbers
$\alpha_{1},\ldots,\alpha_{n}>0$, so that if $t_{i}=\alpha_{i}t_{0}$,
then the products $\left\{ \prod_{i=1}^{n}\hat{\gamma}_{t_{i}}\left(k_{i}\right)\right\} _{\boldsymbol{k}\in\z^{nd}}$
satisfy the above requirements. By (\ref{eq:fourier_coeff_of_gamma_alpha_t}),
for every $\alpha_{i}$ we have 
\begin{equation}
\hat{\gamma}_{\alpha_{i}t_{0}}=\frac{\left(\beta_{t_{0}}+\left(1-\beta_{t_{0}}\right)\hat{\gamma}_{t_{0}}\right)^{\alpha_{i}}-\beta_{t_{0}}^{\alpha_{i}}}{1-\beta_{t_{0}}^{\alpha_{i}}}.\label{eq:uniqueness_of_fourier_expand}
\end{equation}
For a particular $k\in\z^{d}$, let $B_{k}=\left\{ \alpha>0\mid\hat{\gamma}_{\alpha t_{0}}\left(k\right)=0\right\} $
be the set of ``bad'' multipliers for $\hat{\gamma}_{t_{0}}\left(k\right)$.
The coefficient $\hat{\gamma}_{\alpha t_{0}}\left(k\right)$ can be
$0$ only if 
\[
\left(\beta_{t_{0}}+\left(1-\beta_{t_{0}}\right)\hat{\gamma}_{t_{0}}\left(k\right)\right)^{\alpha_{i}}-\beta_{t_{0}}^{\alpha_{i}}=0,
\]
and since $\hat{\gamma}_{t_{0}}\left(k\right)\neq0$, the function
\[
z\mapsto\left(\beta_{t_{0}}+\left(1-\beta_{t_{0}}\right)\hat{\gamma}_{t_{0}}\left(k\right)\right)^{z}-\beta_{t_{0}}^{z}
\]
is a non-constant holomorphic function of $z$. The set of zeros $B_{k}$
is thus isolated, and in particular countable. 

Now let $\boldsymbol{k}\neq\boldsymbol{k}'\in\z^{nd}$ be two different
vectors, and let $B_{\boldsymbol{k},\boldsymbol{k}'}=\left(\union_{i=1}^{n}B_{k_{i}}\right)\bigcup\left(\union_{i=1}^{n}B_{k_{i}'}\right)$
be the set of bad $\alpha_{i}$'s which cause one of the $\hat{\gamma}_{\alpha_{i}t_{0}}$
to be $0$. Let 
\[
R_{\boldsymbol{k,k'}}=\left\{ \boldsymbol{\alpha}\in\r_{>0}^{n}\mid\all j\,\alpha_{j}\notin B_{\boldsymbol{k},\boldsymbol{k}'}\text{ and }\prod_{i=1}^{n}\hat{\gamma}_{\alpha_{i}t_{0}}\left(k_{i}\right)=\prod_{i=1}^{n}\hat{\gamma}_{\alpha_{i}t_{0}}\left(k_{i}'\right)\right\} .
\]
By (\ref{eq:uniqueness_of_fourier_expand}), this means that for every
$\boldsymbol{\alpha}\in R_{\boldsymbol{k},\boldsymbol{k}'}$,
\[
\prod_{i=1}^{n}\frac{\left(\beta_{t_{0}}+\left(1-\beta_{t_{0}}\right)\hat{\gamma}_{t_{0}}\left(k_{i}\right)\right)^{\alpha_{i}}-\beta_{t_{0}}^{\alpha_{i}}}{\left(1-\beta_{t_{0}}^{\alpha_{i}}\right)}-\prod_{i=1}^{n}\frac{\left(\beta_{t_{0}}+\left(1-\beta_{t_{0}}\right)\hat{\gamma}_{t_{0}}\left(k_{i}'\right)\right)^{\alpha_{i}}-\beta_{t_{0}}^{\alpha_{i}}}{\left(1-\beta_{t_{0}}^{\alpha_{i}}\right)}=0.
\]
Let $i^{*}$ be an index so that $k_{i^{*}}\neq k_{i^{*}}'$. Since
none of the factors $\hat{\gamma}_{\alpha_{i}t_{0}}\left(k_{i}\right)$
or $\hat{\gamma}_{\alpha_{i}t_{0}}\left(k_{i}'\right)$ are $0$ by
choice of $R_{\boldsymbol{k},\boldsymbol{k}'}$, we can rearrange
the above, yielding 
\[
\frac{1-\left(1+\frac{1-\beta_{t_{0}}}{\beta_{t_{0}}}\hat{\gamma}_{t_{0}}\left(k_{i^{*}}\right)\right)^{\alpha_{i^{*}}}}{1-\left(1+\frac{1-\beta_{t_{0}}}{\beta_{t_{0}}}\hat{\gamma}_{t_{0}}\left(k_{i^{*}}'\right)\right)^{\alpha_{i^{*}}}}-\prod_{i\neq i^{*}}\frac{\left(\beta_{t_{0}}+\left(1-\beta_{t_{0}}\right)\hat{\gamma}_{t_{0}}\left(k_{i}'\right)\right)^{\alpha_{i}}-\beta_{t_{0}}^{\alpha_{i}}}{\left(\beta_{t_{0}}+\left(1-\beta_{t_{0}}\right)\hat{\gamma}_{t_{0}}\left(k_{i}\right)\right)^{\alpha_{i}}-\beta_{t_{0}}^{\alpha_{i}}}=0.
\]
For fixed $\left\{ \alpha_{i}\right\} _{i\neq i^{*}}$, the expression
on the left-hand side is a function of the form $z\mapsto\frac{1-a^{z}}{1-b^{z}}-c$.
Since $\left\{ \hat{\gamma}_{t_{0}}\left(k\right)\right\} _{k\in\z^{d}}$
are all distinct, $a\neq b$. Thus this map is a non-constant holomorphic
function on $\c\backslash B_{\boldsymbol{k},\boldsymbol{k}'}$, and
so for any fixed choice of $\left\{ \alpha_{i}\right\} _{i\neq i^{*}}$,
has only countably many zeros, i.e only countably many choices for
$\alpha_{i^{*}}$. The Lebesgue measure of $R_{\boldsymbol{k,k'}}$
in $\r^{n}$ is therefore $0$. But the set of all $\alpha_{1},\ldots,\alpha_{n}>0$
such that either there are two equal nonzero products in $\left\{ \prod_{i=1}^{n}\hat{\gamma}_{\alpha_{i}t_{0}}\left(k_{i}\right)\right\} _{\boldsymbol{k}\in\z^{nd}}$
or one of the products is itself $0$ is the countable union 
\[
\bigcup_{\boldsymbol{k},\boldsymbol{k'}\in\z^{nd}}R_{\boldsymbol{k},\boldsymbol{k'}}\union\left(\bigcup_{\boldsymbol{k},\boldsymbol{k}'\in\z^{nd}}\left\{ \left(\alpha_{1},\ldots,\alpha_{k}\right)\mid\alpha_{i}\in\left(B_{k_{i}}\union B_{k_{i}'}\right)\right\} \right),
\]
and so too has measure $0$. In particular, there must exist $\alpha_{1},\ldots,\alpha_{n}>0$
so that the products are distinct and nonzero, as needed.
\end{proof}
\begin{proof}[Proof of Theorem \ref{thm:general_reconstruction}]
We now show that knowledge of $S_{n}\left(\boldsymbol{y}\right)$
is enough to reconstruct $f$ up to a translation of the torus. The
main idea is this. Suppose that $S_{n}\left(\boldsymbol{y}\right)>0$
for some given $\boldsymbol{y}\in\r^{nd}$. Then $\int_{\mathbb{T}^{d}}\prod_{k=0}^{n}f\left(x+\sum_{i=1}^{k}y_{i}\right)dx>0$,
which means that there exists a point $x_{0}\in\mathbb{T}^{d}$ such
that $f\left(x_{0}+\sum_{i=1}^{k}y_{i}\right)=1$ for all $k=0,\ldots,n$.
By considering only $y_{i}$'s which partition $\mathbb{T}^{d}$ into
a grid, we can get an approximation of $\Omega$ by taking the union
of the grid blocks.

Let $\delta_{m}=2\pi/m$, let $n=m^{d}$, and consider the set $\mathcal{Y}=\left\{ \boldsymbol{y}\in\delta\cdot\n^{nd}\mid S_{n}\left(\boldsymbol{y}\right)>0\right\} $.
Note that $\left(2\pi\right)^{d}S_{n}\left(\boldsymbol{0}\right)=\int_{\mathbb{T}^{d}}f\left(x\right)dx=\mu\left(\Omega\right)$
(where $\mu$ is the Lebesgue measure), so if $\mathcal{Y}=\left\{ \emptyset\right\} $
then $\Omega=\emptyset$, and $f$ is identically zero. We may therefore
assume that $\mathcal{Y\neq\emptyset}$.

Each vector $\boldsymbol{y}\in\mathcal{Y}$ defines a set of points
$G_{\boldsymbol{y}}=\left\{ \sum_{i=1}^{k}y_{i}\mid k=0,\ldots,n\right\} $,
where the sum $\sum_{i=1}^{k}y_{i}$ is taken to be in the torus $\mathbb{T}^{d}$.
Since $\delta_{m}$ divides the side-length of the torus, $G_{\boldsymbol{y}}$
can be viewed as a subset of a $d$-dimensional grid in $\mathbb{T}^{d}$,
with the individual $y_{i}$ serving as ``pointer vectors'' to the
next point in the grid. The number of points in $G_{\boldsymbol{y}}$
depends on $\boldsymbol{y}$: If $\boldsymbol{y}=0$, for example,
then $G_{\boldsymbol{y}}=\left\{ 0\right\} $; however, we can also
choose $\boldsymbol{y}$ such that $G_{\boldsymbol{y}}=\delta\z^{d}\intersect\mathbb{T}^{d}$.

Let $\boldsymbol{y}^{*}\in\mathcal{Y}$ be such that $\abs{G_{\boldsymbol{y}*}}\geq\abs{G_{\boldsymbol{y}}}$
for all $\boldsymbol{y}\in\mathcal{Y}$, and let $G_{m}^{*}=G_{\boldsymbol{y}^{*}}$,
so that $G_{m}^{*}$ is a largest possible subset when the pointer
vectors are taken from $\mathcal{Y}$. Using $G_{m}^{*}$, we can
now define a rough, shifted approximation $\Omega_{m}$ to the domain
$\Omega$: Letting $C_{d}=\left[-\frac{1}{2},\frac{1}{2}\right]^{d}$
be the unit $d$-dimensional cube, we cover each point $x\in G_{m}^{*}$
by the scaled cube $\delta_{m}C_{d}$:
\[
\Omega_{m}=G_{m}^{*}+\delta_{m}C_{d}
\]
(here we use the Minkowski sum for the addition of two sets / the
addition of a point and a set). See Figure \ref{fig:omega_approximation}
for an example of this procedure in $2$ dimensions. 

\begin{figure}[H]
\begin{centering}
\begin{minipage}[t]{0.45\columnwidth}%
\begin{center}
\includegraphics[scale=0.33]{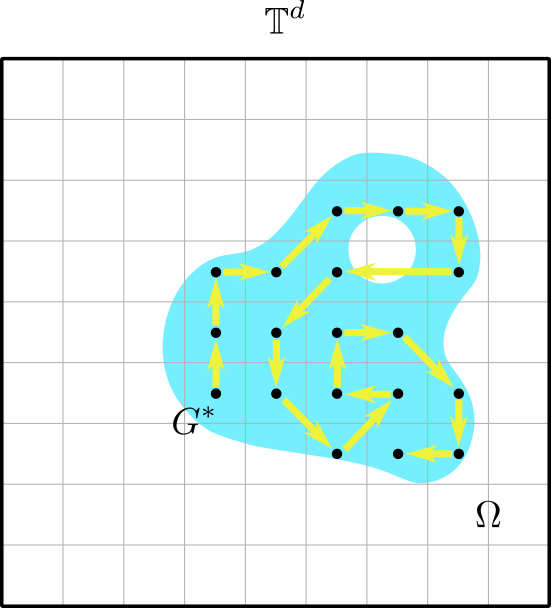}
\par\end{center}%
\end{minipage}\hfill{}%
\begin{minipage}[t]{0.45\columnwidth}%
\begin{center}
\includegraphics[scale=0.33]{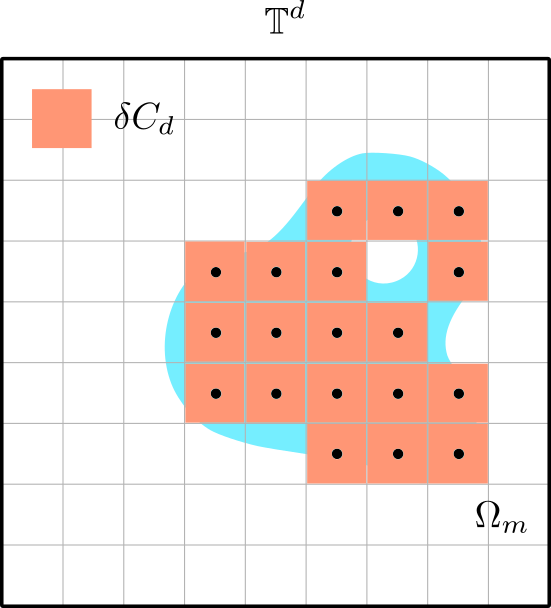}
\par\end{center}%
\end{minipage}
\par\end{centering}
\caption{\textbf{Left}: The domain $\Omega$ (cyan), together with a maximal
grid $G_{m}^{*}$ (black dots). The origin is shifted so that all
grid points fall in $\Omega$. The yellow vectors represent a possible
choice for $y_{i}$. \textbf{Right}: The resultant approximation $\Omega_{m}$.\label{fig:omega_approximation}}
\end{figure}

We now claim that $\Omega_{m}\to\Omega$ up to translations, in the
sense that there is a shift $\theta\in\mathbb{T}^{d}$ such that the
symmetric difference vanishes: $\mu\left(\left(\Omega_{m}+\theta\right)\Delta\Omega\right)\to0$
as $m\to\infty$. To see this, let's look separately at the contribution
of $\left(\Omega_{m}+\theta\right)\backslash\Omega$ and the contribution
of $\Omega\backslash\left(\Omega_{m}+\theta\right)$. First, as mentioned
above, since $S_{n}\left(\boldsymbol{y}\right)>0$, there is a point
$x_{m}$ such that $G_{m}^{*}+x_{m}\subseteq\Omega$. Adding $\delta_{m}C_{d}$
to both sides gives $\Omega_{m}+x_{m}\subseteq\Omega+\delta_{m}C_{d}$.
We then have 
\[
\mu\left(\left(\Omega_{m}+x_{m}\right)\backslash\Omega\right)\leq\mu\left(\left(\Omega+\delta_{m}C_{d}\right)\backslash\Omega\right)\leq\mu\left(\partial\Omega+\delta_{m}C_{d}\right).
\]
The latter expression goes to $0$ as $\delta_{m}\to0$, since 
\[
\lim_{m\to\infty}\mu\left(\partial\Omega+\delta_{m}C_{d}\right)=\mu\left(\bigcap_{m}\left(\partial\Omega+\delta_{m}C_{d}\right)\right)\stackrel{\text{\ensuremath{\partial\Omega} is closed}}{=}\mu\left(\partial\Omega\right)=0.
\]
Suppose now that $x\in\Omega\backslash\left(\Omega_{m}+x_{m}\right)$,
and let $z\in\delta_{m}\mathbb{Z}^{d}\intersect\mathbb{T}^{d}$ be
a grid-point closest to $x-x_{m}$. The point $z$ cannot be in $G_{m}^{*}$:
If it were, then the cube $z+\delta_{m}C_{d}$ (which contains $x-x_{m}$)
would be contained in $\Omega_{m}$, contradicting the fact that $x\notin\Omega_{m}+x_{m}$.
Since $\abs{G_{m}^{*}}$ is maximal, we necessarily have $z+x_{m}\notin\Omega$
(otherwise we could add it to $G_{m}^{*}$). So every point not in
$\Omega_{m}+x_{m}$ can be covered by placing the cube $\delta_{m}C_{d}$
on some point in $\left(\Omega+x_{m}\right)^{c}$. Thus
\begin{align*}
\mu\left(\Omega\backslash\left(\Omega_{m}+x_{m}\right)\right) & =\mu\left(\left(\Omega_{m}+x_{m}\right)^{c}\backslash\Omega^{c}\right)\\
 & \leq\mu\left(\left(\Omega^{c}+\delta_{m}C_{d}\right)\backslash\Omega^{c}\right),
\end{align*}
and again the latter goes to $0$ as $\delta_{m}\to0$. 

We thus have a sequence of vectors $x_{m}\in\t^{d}$ such that $\mu\left(\left(\Omega_{m}+x_{m}\right)\Delta\Omega\right)\to0$
as $m\to\infty$. Since $\t^{d}$ is compact, $x_{m}$ has a subsequence
converging to some $\theta$, and it follows that $\mu\left(\left(\Omega_{m}+\theta\right)\Delta\Omega\right)\to0$
as well.
\end{proof}

\section{Symmetric reconstruction\label{sec:symmetric_reconstruction}}

The proof of Theorem \ref{thm:symmetric_reconstruction} is similar
to that of Theorem \ref{thm:general_reconstruction}. The main difference
is that since $\hat{\gamma}_{t}\left(k\right)=\hat{\gamma}_{t}\left(-k\right)$
for all $k$, we cannot immediately use Lemma \ref{lem:infinite_vandemonde}
to recover $S_{n}\left(\boldsymbol{y}\right)$ anymore. This can be
overcome by working with a completely symmetric version of $S_{n}\left(\boldsymbol{y}\right)$,
denoted $\sigma_{n}\left(\boldsymbol{y}\right)$ and defined as 
\[
\sigma_{n}\left(\boldsymbol{y}\right)=\sum_{\boldsymbol{\eps}\in\left\{ -1,1\right\} ^{n}}S_{n}\left(\eps_{1}y_{1},\ldots,\eps_{n}y_{n}\right).
\]

\begin{prop}
Under the conditions of Theorem \ref{thm:symmetric_reconstruction},
$\sigma_{n}\left(\boldsymbol{y}\right)$ can be calculated from $f\left(X_{t}\right)$
with probability $1$.
\end{prop}

\begin{proof}
The proof uses the same techniques as that of Proposition \ref{prop:reconstructing_s_n};
we highlight the differences here. Starting with the temporal-spatial
relation (\ref{eq:temporal_spatial_relation}),

\begin{align*}
T_{n}\left(\boldsymbol{t}\right) & =\int_{\t^{n}}\left(\prod_{i=1}^{n}D_{t_{i}}\left(y_{i}\right)\right)S_{n}\left(y_{1},\ldots,y_{n}\right)\boldsymbol{dy},
\end{align*}
observe that each integral of the form $\int_{-\pi}^{\pi}D_{t_{i}}\left(y_{i}\right)S_{n}\left(y_{1},\ldots,y_{n}\right)dy_{i}$
can be split into two parts:
\[
\int_{-\pi}^{\pi}D_{t_{i}}\left(y_{i}\right)S_{n}\left(y_{1},\ldots,y_{n}\right)dy_{i}=\int_{-\pi}^{0}D_{t_{i}}\left(y_{i}\right)S_{n}\left(y_{1},\ldots,y_{n}\right)dy_{i}+\int_{0}^{\pi}D_{t_{i}}\left(y_{i}\right)S_{n}\left(y_{1},\ldots,y_{n}\right)dy_{i},
\]
where we use the convention that $\int_{-\pi}^{0}\delta\left(y\right)g\left(y\right)dy=\frac{1}{2}\lim_{\eps\to0^{+}}\int_{-\pi}^{\eps}\delta\left(y\right)g\left(y\right)dy=\frac{1}{2}g\left(0\right)$.
Making the change of variables $y_{i}\to-y_{i}$ in the first integral
and using the fact that $D_{t}$ is symmetric, we thus have 

\begin{align*}
\int_{-\pi}^{\pi}D_{t_{i}}\left(y_{i}\right)S_{n}\left(y_{1},\ldots,y_{n}\right)dy_{i} & =\frac{1}{2}\int_{-\pi}^{\pi}D_{t_{i}}\left(y_{i}\right)\left[S_{n}\left(y_{1},\ldots,-y_{i},\ldots,y_{n}\right)+S_{n}\left(y_{1},\ldots,y_{i},\ldots,y_{n}\right)\right]dy_{i}.
\end{align*}
Performing this $n$ times yields 
\begin{align}
T_{n}\left(\boldsymbol{t}\right) & =\frac{1}{2^{n}}\int_{\t^{n}}\left(\prod_{i=1}^{n}D_{t_{i}}\left(y_{i}\right)\right)\sigma_{n}\left(\boldsymbol{y}\right)\boldsymbol{dy}.\label{eq:symmetric_integral_equation}
\end{align}
As in the proof of Proposition \ref{prop:reconstructing_monomials},
this allows us to calculate the sum 
\[
\sum_{\boldsymbol{k}\in\z^{n}}\left(\prod_{i=1}^{n}\hat{\gamma}_{t_{i}}\left(k_{i}\right)\right)^{m}\hat{\sigma}_{n}\left(\boldsymbol{k}\right)
\]
for all $n,m\in\n$ and every set of times $t_{1},\ldots,t_{n}>0$.
Now, both $\left(\prod_{i=1}^{n}\gamma_{t_{i}}\left(y_{i}\right)\right)$
and $\sigma_{n}\left(\boldsymbol{y}\right)$ are symmetric in every
coordinate $y_{i}$, and so the Fourier coefficients are invariant
under flipping of individual entries. We can therefore restrict the
sum to non-negative $\boldsymbol{k}$ vectors: Defining 
\[
\alpha_{n}\left(\boldsymbol{k}\right)=2^{\#\left\{ i\mid k_{i}\neq0\right\} }\hat{\sigma}_{n}\left(\boldsymbol{k}\right),
\]
we can calculate the sum
\[
\sum_{\boldsymbol{k}\in\z_{+}^{n}}\left(\prod_{i=1}^{n}\hat{\gamma}_{t_{i}}\left(k_{i}\right)\right)^{m}\alpha_{n}\left(\boldsymbol{k}\right)
\]
for all $n,m\in\n$ and times $t_{1},\ldots,t_{n}>0$. As in the proof
of Proposition \ref{prop:reconstructing_s_n}, $\alpha_{n}$ (and
therefore $\sigma_{n}$) can be recovered from these quantities using
Lemma \ref{lem:infinite_vandemonde}, since by assumption there is
a time $t_{0}$ such that $\left\{ \hat{\gamma}_{t_{0}}\left(k\right)\right\} _{k\geq0}$
are all distinct and non-zero. 
\end{proof}
\begin{prop}
Let $\delta>0$ divide $2\pi$. Given $\sigma_{n}\left(\boldsymbol{y}\right)$
for all $n\in\n$ and $\boldsymbol{y}\in\t^{n}$, it is possible to
calculate $S_{n}\left(\boldsymbol{y}\right)+S_{n}\left(-\boldsymbol{y}\right)$
for all $n\in\n$ and all $\boldsymbol{y}\in\left\{ \left(k_{1}\delta,\ldots,k_{n}\delta\right)\mid k_{i}\in\z_{+}\right\} $.
\end{prop}

\begin{proof}
The proof is by induction. For $n=1$, we just have $\sigma_{1}\left(y\right)=S_{1}\left(y\right)+S_{1}\left(-y\right)$,
and for general $n$ and $\boldsymbol{k}=\boldsymbol{0}$, we just
have $\sigma_{n}\left(\boldsymbol{0}\right)=2^{n}S_{n}\left(\boldsymbol{0}\right)$.
Now let $n\in\n$ and $\boldsymbol{k}\in\z_{+}^{n}$. Assume that
the statement holds true for $n$ for all vectors $\boldsymbol{k}'$
with $\sum_{i=1}^{n}k_{i}'<\sum_{i=1}^{n}k_{i}$, and also for all
$m<n$. We have 
\begin{align*}
\sigma_{n}\left(k_{1}\delta,\ldots,k_{n}\delta\right) & =\sum_{\boldsymbol{\eps}\in\left\{ -1,1\right\} ^{n}}S_{n}\left(\eps_{1}k_{1}\delta,\ldots,\eps_{n}k_{n}\delta\right)\\
 & =S_{n}\left(k_{1}\delta,\ldots,k_{n}\delta\right)+S_{n}\left(-k_{1}\delta,\ldots,-k_{n}\delta\right)+\sum_{\boldsymbol{\eps}\in\left\{ -1,1\right\} ^{n},\,\eps_{i}\text{ not all equal}}S_{n}\left(\ldots\right).
\end{align*}
If not all $\eps_{i}$ are equal, then $S_{n}\left(\eps_{1}k_{1}\delta,\ldots,\eps_{n}k_{n}\delta\right)$
is equal to some $S_{m}\left(k'_{1}\delta,\ldots,k_{m}'\delta\right)$
for $m\leq n$: The total sum $\sum_{i}\eps_{i}k_{i}$ is strictly
smaller than $\sum_{i}k_{i}$, and so the partial sums can be seen
as the forward-only jumps of some $\boldsymbol{y}'$ with corresponding
$\boldsymbol{k}'$ such that $\sum_{i}k_{i}'<\sum_{i}k_{i}$ (the
strict $m<n$ case is when the partial sums $\sum\eps_{i}k_{i}$ themselves
are not unique). See Figure \ref{fig:equivalence_of_pointers} for
a visualization. Similarly, $S_{n}\left(-\eps_{1}k_{1}\delta,\ldots,-\eps_{n}k_{n}\delta\right)=-S_{m}\left(-k'_{1}\delta,\ldots,-k_{m}'\delta\right)$.
Since the sum over all $\eps_{i}$ can be split into polar pairs,
by the induction hypothesis we can calculate the sum $\sum_{\boldsymbol{\eps}\in\left\{ -1,1\right\} ^{n},\,\eps_{i}\text{ not all equal}}S_{n}$,
and therefore also $S_{n}\left(k_{1}\delta,\ldots,k_{n}\delta\right)+S_{n}\left(-k_{1}\delta,\ldots,-k_{n}\delta\right).$

\begin{figure}[H]
\begin{centering}
\includegraphics[width=0.33\textheight]{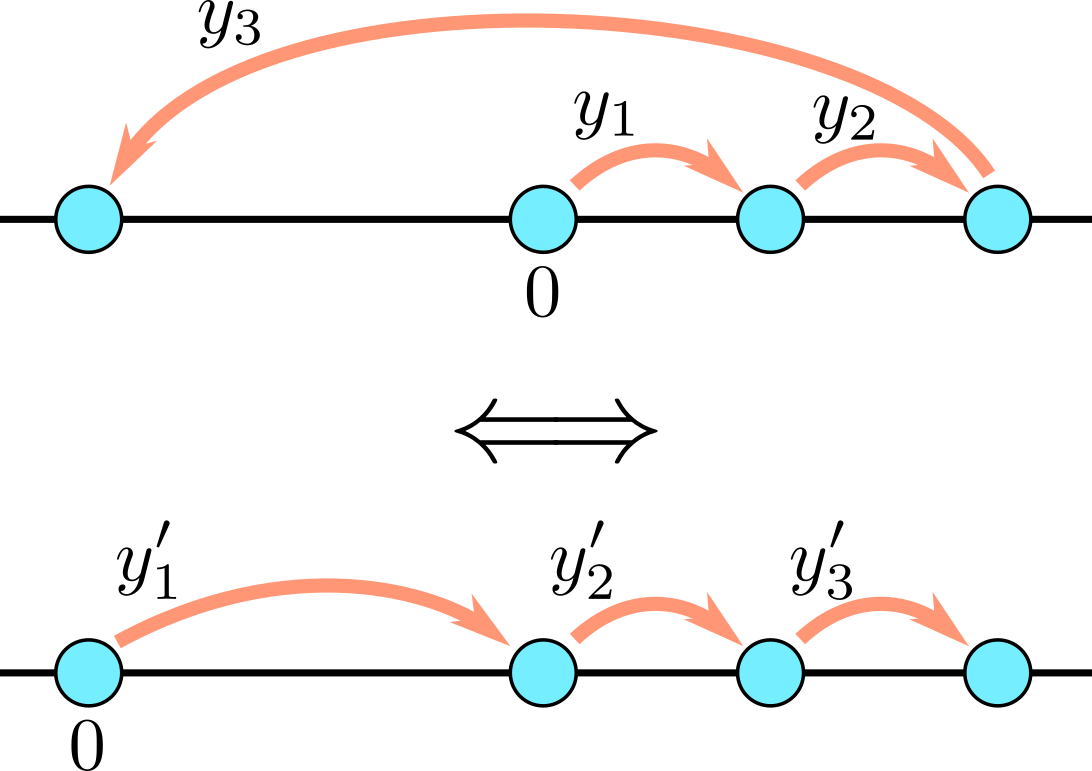}
\par\end{centering}
\caption{$S_{3}\left(y_{1},y_{2},y_{3}\right)=S_{3}\left(y_{1}',y_{2}',y_{3}'\right)$.
The choice of origin does not matter, since $S_{n}$ integrates over
all $\protect\t$ anyway. The total length $y_{1}'+y_{2}'+y_{3}'$
is smaller than $y_{1}+y_{2}+y_{3}$.\label{fig:equivalence_of_pointers}}

\end{figure}
\end{proof}
\begin{proof}[Proof sketch of Theorem \ref{thm:symmetric_reconstruction}]
Similarly to the proof of Theorem \ref{thm:general_reconstruction},
once we know $S_{n}\left(k_{1}\delta,\ldots,k_{n}\delta\right)+S_{n}\left(-k_{1}\delta,\ldots,-k_{n}\delta\right)$,
for every $\delta_{n}=2\pi/n$ we can construct a set $\Omega_{n}$
according a $\boldsymbol{y}\in\left\{ \boldsymbol{y}\in\delta\cdot\n^{n}\mid S_{n}\left(\boldsymbol{y}\right)+S_{n}\left(-\boldsymbol{y}\right)>0\right\} $
which maximizes $G_{\boldsymbol{y}}$. The resulting $f_{n}$ contains
a subsequence which converges to either $f\left(x\right)$ or $f\left(-x\right)$,
where the ambiguity is because we do not know which of the two of
$S_{n}\left(\boldsymbol{y}\right)$ and $S_{n}\left(-\boldsymbol{y}\right)$
was greater than $0$.
\end{proof}

\section{Example: Brownian motion\label{sec:example_brownian_motion}}

\subsection{Proof of Corollary \ref{cor:brownian_motion_can_feel}}
\begin{proof}
For $d=1$, the step distribution $\gamma_{t}$ of Brownian motion
on $\t$ is that of a wrapped normal distribution with drift, and
is given by 
\[
\gamma_{t}\left(y\right)=\frac{1}{\sqrt{2t\pi}}\sum_{m=-\infty}^{\infty}e^{-\left(y-vt+2\pi m\right)^{2}/2t}.
\]
The Fourier coefficients of $\gamma_{t}\left(y\right)$ can readily
be calculated, by observing that the wrap-around gives the continuous
Fourier transform evaluated at integer points:

\begin{align*}
\hat{\gamma}_{t}\left(k\right) & =\int_{-\pi}^{\pi}e^{-iky}\frac{1}{\sqrt{2t\pi}}\sum_{m=-\infty}^{\infty}e^{-\left(y-vt+2\pi m\right)^{2}/2t}dy\\
 & =\int_{-\infty}^{\infty}e^{-iky}\frac{1}{\sqrt{2t\pi}}e^{-\left(y-vt\right)^{2}/2t}dy\\
 & =e^{-ivtk}e^{-2t\pi^{2}k^{2}}.
\end{align*}
For standard Brownian motion, without drift, the coefficients are
symmetric, but $\left\{ \hat{\gamma}_{t}\left(k\right)\right\} _{k\geq0}$
are all distinct, and so by Theorem \ref{thm:symmetric_reconstruction},
reconstruction is possible using Brownian motion up to rotations and
reflections. 

The drift, however, can make the coefficients distinct. For $d=1$,
any non-zero drift will do, and reconstruction is possible up to rotations.
In the general case, the Fourier coefficients of $\gamma_{t}$ are
\[
\hat{\gamma}_{t}\left(k\right)=e^{-it\sum_{j=1}^{d}v_{j}\cdot k_{j}}e^{-2t\pi^{2}\sum_{j}k_{j}^{2}}\,\,\,\,;\,\,\,k\in\z^{d}.
\]
In order for the coefficients to be distinct, it suffices to make
the factors $e^{-it\sum_{j=1}^{d}v_{j}\cdot k_{j}}$ all distinct,
i.e there should exist a time $t$ such that for every $k\neq k'\in\z^{d}$
and every $m\in\z$,
\[
\sum_{j=1}^{d}v_{j}\left(k_{j}-k_{j}'\right)\neq2\pi m/t.
\]
Choosing $t$ so that $\left\{ v_{1},\ldots,v_{d},\frac{2\pi}{t}\right\} $
are all rationally independent completes the proof.
\end{proof}

\subsection{An explicit inversion}

As a side note, we would like to mention that for Brownian motion,
inverting the symmetric integral (\ref{eq:symmetric_integral_equation})
(note the changed range of integration), 
\[
T_{n}\left(\boldsymbol{t}\right)=\int_{\left[0,\pi\right]^{n}}\left(\prod_{i=1}^{n}\gamma_{t_{i}}\left(y_{i}\right)\right)\sigma_{n}\left(\boldsymbol{y}\right)\boldsymbol{dy},
\]
is possible without resorting to Lemma \ref{lem:infinite_vandemonde}:
There exists an explicit relation between $\sigma_{n}\left(\boldsymbol{y}\right)$
and $T_{n}\left(\boldsymbol{t}\right)$. The relation appears in (e.g)
\cite[3.2-6, item 21]{polyanin_manzhirov_handbook_of_integral_equations}.
We repeat the arguments here for completeness.

If we treat $\sigma_{n}\left(\boldsymbol{y}\right)$ as a periodic
function over all $\r^{n}$ where every coordinate has period $2\pi$,
we can replace the folded normal distribution $\gamma_{t}$ with a
normal distribution over all $\r_{+}$:
\[
T_{n}\left(\boldsymbol{t}\right)=\int_{\boldsymbol{y}\in\r_{+}^{n}}\left(\prod_{i=1}^{n}\frac{1}{\sqrt{2\pi t_{i}}}e^{-\frac{y_{i}^{2}}{2t_{i}}}\right)\sigma_{n}\left(\boldsymbol{y}\right)\boldsymbol{dy}.
\]
Let $\boldsymbol{p}\in\r_{+}^{n}$. Multiplying both sides by $e^{-p_{1}t_{1}-\ldots-p_{n-1}t_{n-1}}$
and integrating all $t_{i}$s from $0$ to $\infty$, we get
\[
\mathcal{L}\left\{ T_{n}\right\} \left(\boldsymbol{p}\right)=\int_{\boldsymbol{t}\in\r_{+}^{n}}\int_{\boldsymbol{y}\in\r_{+}^{n}}\left(\prod_{i=1}^{n}\frac{1}{\sqrt{2\pi t_{i}}}e^{-\frac{y_{i}^{2}}{2t_{i}}}e^{-p_{i}t_{i}}\right)\sigma_{n}\left(\boldsymbol{y}\right)\boldsymbol{dy}\boldsymbol{dt},
\]
where $\mathcal{L}\left\{ f\right\} $ is the Laplace transform of
$f$. The individual integrals over $t_{i}$ in the right hand side
can be readily calculated to be: 
\[
\int_{0}^{\infty}e^{-pt-y^{2}/2t}\frac{1}{\sqrt{2\pi t}}dt=\frac{1}{\sqrt{2p}}e^{-\sqrt{2p}y}.
\]
This gives 
\[
\mathcal{L}\left\{ T_{n}\right\} \left(\boldsymbol{p}\right)=\int_{\r_{+}^{n}}\left(\prod_{i=1}^{n}\frac{1}{\sqrt{2p_{i}}}e^{-\sqrt{2p_{i}}y}\right)\sigma_{n}\left(\boldsymbol{y}\right)\boldsymbol{dy}.
\]
Up to a change of variables $s_{i}=\sqrt{2p_{i}}$, the right hand
side is the Laplace transform of $\sigma_{n}\left(\boldsymbol{y}\right)$.
Thus
\[
\sigma_{n}=\mathcal{L}^{-1}\left\{ \left(\prod_{i=1}^{n}s_{i}\right)\mathcal{L}\left\{ T_{n}\right\} \left(\frac{1}{2}s_{1}^{2},\ldots,\frac{1}{2}s_{n}^{2}\right)\right\} .
\]

\section{Proof of Lemma \ref{lem:infinite_vandemonde}\label{sec:vandermonde_lemma}}
\begin{proof}
As noted in Remark \ref{rem:vandermonde_multiplication_is_ok}, the
equality $V\boldsymbol{x}=0$ means that for every index $i\in\n$
we have 
\begin{equation}
0=\sum_{j=1}^{\infty}V_{ij}x_{j}=\sum_{j=1}^{\infty}z_{j}^{i}x_{j}.\label{eq:sum_equals_zero}
\end{equation}
Since $\boldsymbol{z}\in\ell^{p}\left(\c\right)$ and $\boldsymbol{x}\in\ell^{q}\left(\c\right)$,
by Hölder's inequality the series $\left(z_{j}^{i}x_{j}\right)_{j}$
is absolutely convergent for all $i$, and we can change the order
of summation; without loss of generality we can assume that $z_{n}$
are ordered so that $\abs{z_{1}}\geq\abs{z_{2}}\geq\ldots$. 

Assume by induction that $x_{1},\ldots,x_{k}$ have already been shown
to be equal to $0$, and let $\ell>0$ be such that $\abs{z_{k+1}},\ldots,\abs{z_{k+\ell}}$
are all of equal magnitudes, but $\abs{z_{k+\ell}}>\abs{z_{k+\ell+1}}$.
Since $z_{n}\to0$, $\ell$ is necessarily finite. For any fixed $i$,
the sum (\ref{eq:sum_equals_zero}) can then be split into two parts:
\[
\sum_{s=1}^{\ell}z_{k+s}^{i}x_{k+s}+\sum_{j>k+\ell}z_{j}^{i}x_{j}=0.
\]
Dividing by $z_{k+\ell}^{i}$ and denoting $\omega_{s}=\frac{z_{k+s}}{z_{k+\ell}}$,
we have

\[
\sum_{s=1}^{\ell}\omega_{s}^{i}x_{k+s}+\sum_{j>k+\ell}\left(\frac{z_{j}}{z_{k+\ell}}\right)^{i}x_{j}=0.
\]
Consider the $\ell$ equations of the above form for $i=r\cdot m$,
where $r=1,\ldots,\ell$ and $m$ is a large number. In matrix form,
this system of $\ell$ equations can be written as 
\begin{equation}
V_{m}\tilde{\boldsymbol{x}}+\boldsymbol{u}=0,\label{eq:finite_vandermonde}
\end{equation}
where:
\begin{enumerate}
\item $V_{m}$ is a finite $\ell\times\ell$ Vandermonde matrix with generators
$\omega_{s}^{m}=\left(\frac{z_{k+s}}{z_{k+\ell}}\right)^{m}$, $s=1,\ldots,\ell$.
\item $\tilde{\boldsymbol{x}}$ is a vector of size $\ell$ with $\tilde{x}_{s}=x_{k+s}$
for $s=1,\ldots,\ell$.
\item \label{enu:vandermonde_u_has_bounded_entries}$\boldsymbol{u}$ is
a vector of length $\ell$ with entries $u_{r}=\sum_{j>k+\ell}\left(\frac{z_{j}}{z_{k+\ell}}\right)^{rm}x_{j}$
for $r=1,\ldots,\ell$. By our ordering and choice of $\ell$, $\abs{z_{k+\ell}}>\abs{z_{j}}$
for $j>k+\ell$, and we can factor out an exponentially decreasing
term from each summand:
\begin{align*}
\abs{u_{r}} & \leq\abs{\frac{z_{k+\ell+1}}{z_{k+\ell}}}^{rm}\sum_{j>k+\ell}\abs{\frac{z_{j}}{z_{k+\ell+1}}}^{rm}\abs{x_{j}}\\
\left(\text{Hölder's inequality}\right) & \leq\abs{\frac{z_{k+\ell+1}}{z_{k+\ell}}}^{rm}\left(\sum_{j>k+\ell}\abs{\frac{z_{j}}{z_{k+\ell+1}}}^{prm}\right)^{1/p}\norm{\boldsymbol{x}}_{q}.
\end{align*}
The sum $\sum_{j>k+\ell}\abs{\frac{z_{j}}{z_{k+\ell+1}}}^{prm}$ is
finite since $\boldsymbol{z}\in\ell^{p}\left(\c\right)$, and is in
fact uniformly bounded as a function of $m$, since every summand
has magnitude less than or equal to $1$. Since $\abs{z_{k+\ell+1}}<\abs{z_{k+\ell}}$,
the term $\abs{\frac{z_{k+\ell+1}}{z_{k+\ell}}}^{rm}$ goes to $0$
as $m\to\infty$, and so the entries of $\boldsymbol{u}$ also decrease
to $0$ as $m\to\infty$.
\end{enumerate}
The generators of $V_{m}$ are all distinct, and so $V_{m}$ is invertible.
Equation (\ref{eq:finite_vandermonde}) thus gives 
\[
\tilde{\boldsymbol{x}}=-V_{m}^{-1}\boldsymbol{u}.
\]
As mentioned in item (\ref{enu:vandermonde_u_has_bounded_entries}),
the entries of $\boldsymbol{u}$ decay to $0$ in $m$. To show that
$\boldsymbol{\tilde{x}}=0$ (and thus all of $x_{k+s}=0$ for $s=1,\ldots,\ell$),
it therefore suffices to uniformly bound the infinity-norm $\norm{V_{m}^{-1}}_{\infty}$
for infinitely many $m$. Recall that the inverse of the $\ell\times\ell$
Vandermonde matrix with generators $\omega_{s}^{m}$ has entries 
\begin{equation}
\left(V_{m}^{-1}\right)_{ij}=\left(-1\right)^{j-1}\dfrac{C_{ij}}{\omega_{i}^{m}\prod_{\stackrel{1\leq k\leq\ell}{k\neq i}}\left(\omega_{i}^{m}-\omega_{k}^{m}\right)},\label{eq:vandemonde_inverse_entries}
\end{equation}
where
\[
C_{ij}=\begin{cases}
\sum_{\stackrel{1\leq k_{1}<\ldots<k_{n-j}\leq\ell}{m_{1},\ldots,m_{\ell-j}\neq i}}\omega_{m_{1}}\cdots\omega_{m_{\ell-j}} & 1\leq j<\ell\\
1 & j=\ell.
\end{cases}
\]
Since the generators $\omega_{s}^{m}$ are all on the unit circle,
the numerator $C_{ij}$ is always bounded by a constant independent
of $m$. For the denominator, it suffices to show that we can choose
infinitely many $m$ such that $\abs{\omega_{i}^{m}-\omega_{j}^{m}}$
is uniformly bounded away from $0$ for all $i\neq j$.
\begin{lem}[Recurrent rotations]
\label{lem:recurrent_rotations}Let $\ell>0$ be an integer and let
$\eps>0$. There exists a constant $C_{\eps,\ell}$ such that for
every set of distinct points $\omega_{1},\ldots,\omega_{\ell}$ on
the unit circle, there is an integer $1\leq m\leq C_{\eps,\ell}$
such that 

\[
\frac{1}{\pi}\abs{\arg\left(\omega_{i}^{m}\right)}\leq\eps\,\,\,\,i=1,\ldots,\ell,
\]
where $\arg:\c\to\left(-\pi,\pi\right]$ returns the angle with the
origin.
\end{lem}

\begin{proof}
The proof is by induction on $\ell$. For $\ell=1$, denote $\omega=e^{\pi i\alpha}$
for $\alpha\in\left[-1,1\right]$. If $\abs{\alpha}\leq\eps$, then
$m=1$ will do. Otherwise, let $k>0$ be the smallest integer such
that $\alpha\in\left[-2^{k}\eps,2^{k}\eps\right]$. Then one of the
points $y\in\left\{ \omega,\omega^{2},\ldots,\omega^{\ceil{1/\alpha}+1}\right\} $
satisfies $y=e^{\pi i\beta}$ with $\beta\in\left[-2^{k-1}\eps,2^{k-1}\eps\right]$
: The points $\omega,\omega^{2},\ldots,\omega^{\ceil{1/\alpha}+1}$
make at least one complete revolution around the unit circle, but
since the angle between two consecutive points is at most $2^{k}\eps$,
one of these points must fall in the interval $\left[-2^{k-1}\eps,2^{k-1}\eps\right]$.
Repeating this procedure iteratively, we obtain a sequence of points
$y_{1}=e^{\pi i\beta_{1}},\ldots,y_{q}=e^{\pi i\beta_{q}}$, where:
\begin{enumerate}
\item $y_{1}=\omega=e^{\pi i\alpha}.$
\item $y_{i+1}=y_{i}^{\ell_{i}}$ for some $\ell_{i}\leq\ceil{1/\abs{\beta_{i}}}+1$.
\item $\frac{1}{\pi}\abs{\arg\left(y_{q}\right)}\leq\eps$
\item The number of iterations $q$ is bounded by $\log_{2}1/\eps+1$.
\end{enumerate}
Apart from $\beta_{q}$, the magnitude of each $\beta_{i}$ is larger
than $\eps$. Thus each $\ell_{i}\leq\ceil{1/\eps}+1$, and we have
\[
\eps\geq\frac{1}{\pi}\abs{\arg\left(y_{q}\right)}=\frac{1}{\pi}\abs{\arg\omega^{\ell_{1}^{\ldots^{\ell_{q-1}}}}}=\frac{1}{\pi}\abs{\arg\omega^{\ell_{1}\cdot\ldots\cdot\ell_{q-1}}},
\]
yielding an $m\leq\prod_{j=1}^{q-1}\ell_{j}\leq\left(\ceil{1/\eps}+1\right)^{\log_{2}1/\eps}$.
We therefore have $C_{\eps,1}=\left(\ceil{1/\eps}+1\right)^{\log_{2}1/\eps}$. 

Now assume by induction that the statement holds for all $n<\ell$.
By applying the induction hypothesis on the first $\ell-1$ points
with $\eps'=\eps/C_{\eps,1}$, we obtain an integer $m_{1}\leq C_{\eps',\ell-1}$
such that for $i=1,\ldots,\ell-1$, 
\begin{equation}
\frac{1}{\pi}\abs{\arg\left(\omega_{i}^{m_{1}}\right)}\leq\frac{\eps}{C_{\eps,1}}.\label{eq:first_l-1_points_ok}
\end{equation}
This does not give any bound on $\arg\left(\omega_{\ell}^{m_{1}}\right)$.
However, we can now apply the lemma for a single point $\omega_{\ell}^{m_{1}}$
and $\eps$, obtaining an $m_{2}\leq C_{\eps,1}$ such that 
\[
\frac{1}{\pi}\abs{\arg\left(\left(\omega_{\ell}^{m_{1}}\right)^{m_{2}}\right)}\leq\eps.
\]
Choosing $m=m_{1}m_{2}\leq C_{\eps,1}C_{\eps',\ell-1}$ yields the
required bound on $\omega_{\ell}$; as for $i<\ell,$ using (\ref{eq:first_l-1_points_ok}),
we have
\[
\frac{1}{\pi}\abs{\arg\left(\left(\omega_{i}^{m_{1}}\right)^{m_{2}}\right)}\leq\frac{\eps}{C_{\eps,1}}\cdot m_{2}\leq\eps
\]
as well.
\end{proof}
We can now finish the proof of Lemma \ref{lem:infinite_vandemonde}.
If all the ratios $\left\{ \arg\left(\omega_{s}\right)/\pi\right\} _{s=1}^{\ell}$
are rational, then there are infinitely many $m$'s such that 
\[
\omega_{s}^{m}=\omega_{s}
\]
for all $s=1,\ldots,\ell$. The denominator in (\ref{eq:vandemonde_inverse_entries})
stays the same in this case for all such $m$. Otherwise, there is
an $s^{*}$ such that $\arg\left(\omega_{s^{*}}\right)$ is an irrational
multiple of $\pi$. Let $\eps_{n}\to0$ be a positive sequence and
let $m_{n}$ be the number of rotations given by Lemma \ref{lem:recurrent_rotations}
applied to $\left\{ \omega_{s}\right\} _{s=1}^{\ell}$ with $\eps_{n}$.
Then $m_{n}$ has a subsequence which diverges to infinity, since
for any finite set of values of $m$, $\arg\left(\omega_{s^{*}}^{m}\right)$
is bounded below. For all large enough $n$, we necessarily have $\abs{\omega_{i}^{m_{n}+1}-\omega_{j}^{m_{n}+1}}\geq\frac{1}{2}\abs{\omega_{i}-\omega_{j}}$
for all $i\neq j$, and so the denominator in (\ref{eq:vandemonde_inverse_entries})
is bounded.
\end{proof}

\section{Other directions and open questions\label{sec:open_questions}}

At least for processes with continuous paths on the circle $\t$ (such
as Brownian motion), it seems reasonable that it's possible to reconstruct
functions which are more complicated than indicators. 
\begin{question}
What classes of functions are reconstructible from Brownian motion,
with or without drift? Given a nice enough function $f:\t\to\r$,
is it perhaps possible to stitch together its level sets $\left\{ f\geq\alpha\right\} $,
which we know are reconstructible due to Theorems \ref{thm:general_reconstruction}
/ \ref{thm:symmetric_reconstruction}, to gain knowledge about the
entire function?
\end{question}

\begin{question}
Find another algorithm which shows directly that Brownian motion can
reconstruct $\mathcal{F}_{d}$, using its local properties. 
\end{question}

Theorem \ref{thm:symmetric_reconstruction} suggests that stochastic
processes with symmetries should allow reconstruction, up to some
symmetry of $\t^{d}$ itself. 
\begin{question}
Suppose that some coordinates of $X_{t}$ are independent from others,
or that $X_{t}$ is invariant under some orthogonal transformation
in $O_{\r}\left(d\right)$. What can be said about reconstruction
on $\t^{d}$?
\end{question}

It is natural to consider larger spaces of functions, larger spaces,
and more general distributions.
\begin{question}
What can be said for sets $\Omega$ with fat boundary, i.e $\mu\left(\partial\Omega\right)>0$?
\end{question}

\begin{question}
Can Theorem \ref{thm:general_reconstruction} be extended to general
compact Riemannian manifolds?
\end{question}

\begin{question}
How do the results extend to processes whose step sizes are allowed
to contain atoms at $x\neq0$? (consider, for example, a Poisson process
which, when its clock fires, jumps either by $\alpha_{1}$ or $\alpha_{2}$,
with $\left\{ \alpha_{1},\alpha_{2},\pi\right\} $ rationally independent).
\end{question}

Finally, the independence condition of the drift in Corollary \ref{cor:brownian_motion_can_feel}
gives rise to a slightly different model for reconstruction on the
torus, where we try to learn $f$ from its values on a random (irrational)
geodesic.
\begin{question}
Let $v$ be a uniformly random unit vector in $\r^{d}$. Which classes
of functions $f$ can be reconstructed (with probability 1) from $f\left(t\cdot v\right)_{t\in\r}$?
How about $f\left(X_{t}\right)$, where $X_{t}$ is Brownian motion
with random drift $v_{t}$?
\end{question}

\bibliographystyle{plain}
\bibliography{brownian_motion_can_feel_the_shape_of_a_drum}

\end{document}